\documentclass[a4paper,11pt]{article}
\usepackage{vmargin}
\usepackage{amssymb}
\usepackage{amsthm}
\usepackage{amsmath}
\usepackage{amsfonts}
\usepackage{indentfirst}
\usepackage{comment}
\usepackage{setspace}  
\usepackage{anysize}
\usepackage{subfigure}
\usepackage[all]{xy}
\usepackage{caption}
\usepackage[square,sort,comma,numbers]{natbib}
\usepackage{floatrow}
\usepackage{pgf,tikz,graphics}
\usetikzlibrary{arrows, automata}
\newfloatcommand{capbtabbox}{table}[][\FBwidth]
\setpapersize{A4} 
\setmarginsrb{2cm}{1.5cm}{1cm}{1cm}{1pt}{27pt}{3pt}{27pt}

\newcommand{\Z}{\ensuremath{\mathbb{Z}_2^n}}

\theoremstyle{plain}
\newtheorem{theorem}{Theorem}[section]

\newtheorem{corollary}[theorem]{Corollary}
\newtheorem{proposition}[theorem]{Proposition}
\theoremstyle{definition}
\newtheorem{definition}[theorem]{Definition}
\newtheorem{example}[theorem]{Example}
\theoremstyle{remark}

\newtheorem{conjecture}[theorem]{Conjecture}

\title{Spin structures on generalized real Bott manifolds}

\begin{document}

\maketitle

\begin{center}
	\textbf{Asl\i \ G\"U\c{C}L\"UKAN \.{I}LHAN, Sabri Kaan G\"{U}RB\"{U}ZER, Semra PAMUK}
\end{center}

\begin{abstract}In this paper, we give a necessary and sufficient condition for a generalized real Bott manifold to have a Spin structure in terms of column vectors of the associated matrix. We also give an interpretation of this result to the associated acyclic $\omega$-weighted digraphs. Using this, we obtain a family of real Bott manifolds that does not admit Spin Structure.
	
\textbf{Keywords:} Generalized Bott manifold, small cover, acyclic digraph
\end{abstract}

2020 {\itshape AMS Mathematics Subject Classification:} 37F20; 57S12

\section{Introduction} \label{sect:Introduction}
A generalized real Bott tower of height $k$ is a sequence of real projective bundles\\
\begin{eqnarray}\label{eq:BGott}	B_k \rightarrow B_{k-1} \rightarrow \cdots \rightarrow B_1 \rightarrow \{pt\}
\end{eqnarray} where $B_i$ is the projectivization of the Whitney sum of $n_i+1$ real line bundles
over $B_{i-1}$. This notion is introduced by Choi, Masuda and Suh \cite{Main} as a generalization of the notion of a Bott tower given in \cite{GrossbergKArshon}. The manifold $B_k$ is called a real Bott manifold when $n_i=1$ for each $i$ and a generalized real Bott manifold, otherwise. The manifold $B_k$ can be realized as a small cover over $\prod_{i=1}^k \Delta^{n_i}$ where $\Delta^{n_i}$ is the $n_i$-simplex \cite[Corollary 4.6]{KurakiLu}. It is also known that every small cover over a product of simplices is a generalized real Bott manifold \cite[Remark 6.5]{Main}.

Let $P$ be a simple convex polytope of dimension $n$ with the facet set $\mathcal{F}(P)=\{F_1,\cdots, F_m\}$. For every small cover $M$ over $P$,  there is an associated $(n \times (m-n))$ matrix $A=[a_{ij}]$ with entries in $\mathbb{Z}_2$ which can be used to reconstruct $M$ (See Section \ref{sect:Prelim}). Moreover, the mod $2$ cohomology ring structure of $M$ depends only on the face poset of $P$ and the matrix $A$. More precisely, let $\mathbb{Z}_2[P]$ be the  Stanley-Reisner ring of $P$, that is, the quotient of the polynomial ring $\mathbb{Z}_2[x_1,\cdots,x_m]$ with the ideal $I$ generated by the square free monomials $x_{i_1}\cdots x_{i_r}$ for which $F_{i_1}\cap \cdots \cap F_{i_r}$ is empty. There is a graded ring isomorphism between $H^{\ast}(M, \mathbb{Z}_2)$ and $\mathbb{Z}_2[P]/J$ where $J$ is the homogeneous ideal generated by the monomials $\displaystyle x_i+\sum_{j=1}^{m-n}a_{ij}x_{n+j},$ \cite[Theorem 4.14]{DavisJanuszkiewicz}. Here the degree of $x_i$ is $1$. In \cite[Corollary 6.8]{DavisJanuszkiewicz}, Davis and Januskiewicz show that the total Stiefel-Whitney class of $P$ is given by
\begin{eqnarray} \label{eq:total} w(M)=\Big( \prod_{i=1}^{m-n}(1+x_{n+i})\Big) \cdot \Big( \prod_{i=1}^{n}\big(1+\sum_{j=1}^{m-n}a_{ij}x_{n+j}\big)\Big) \ \ \ \mathrm{mod} \ I. \end{eqnarray}
Therefore the coefficient of $x_i$ in the first Stiefel-Whitney class of $M$ is one more than the sum of the entries of the $(i-n)$-th column of $A$ when $i> n$ and zero, otherwise. Hence the small cover $M$ is orientable if and only if the sum of the entries of the $i$-th column of the matrix $A$ is congruent to $1$ modulo $2$ for each $i\geq 1$ \cite[Theorem, 1.7]{NAkayamaNishimura}. Since $A$ is a matrix over $\mathbb{Z}_2$, the sum of the entries of the $j$-th column of $A$ is equivalent to the dot product of the column vector with itself. Therefore the small cover $M$ is orientable if and only if $A_j\cdot A_j \equiv 1$ modulo $2$ where $A_j$ denote the $j$-th column vector of $A$. 

In Section \ref{sect:Prelim}, we observe that a small cover $M$ has a Spin structure when $A_i \cdot A_i\equiv3  \ (\mathrm{mod} \ 4)$ and $A_i \cdot A_j\equiv 0 \ (\mathrm{mod} \ 2)$ for all $1\leq i < j \leq m-n$ (Corollary \ref{cor:spingeneral}). It turns out that when $P$ is a product of simplices of dimensions greater than $1$, the converse is also true (Corollay \ref{cor:seconSWforPS}). In other words, when each $B_i$ is a projectivization of the Whitney sum of $3$ or more line bundles, the generalized Bott manifold $B_k$ has a Spin structure if and only if $A_i \cdot A_i\equiv3  \ (\mathrm{mod} \ 4)$ and $A_i \cdot A_j\equiv 0 \ (\mathrm{mod} \ 2)$ for all $1\leq i < j \leq m-n$. In Theorem \ref{theorem:seconSWforPS}, we give a criterion for an arbitrary generalized Bott manifold $B_k$ to have a Spin structure. It is equivalent to the criterion given in \cite{DsouzaUma}. 

In \cite[Lemma 2.1]{Gasior1}, Gasior gives a formula for the second Stiefel-Whitney class of $M(A)$ in terms of the second Stiefel-Whitney classes of $M(A_{ij}),$ where $A_{ij}$ is an $n\times n$ matrix whose $k$-th column is $A_k$ if $k=i,j$ and $0$ otherwise, called an elementary component. After reducing the problem to elementary components, the author gives a necessary and sufficient condition on existence of a Spin structure on them in \cite[Theorem 1.2]{Gasior1} which can also be obtained as a corollary of Theorem \ref{theorem:seconSWforPS}. Moreover, Proposition \ref{prop:Gasior1} is a generalization of this result to the generalized real Bott manifolds. 

It is well-known that real Bott manifolds can be classified by acyclic digraphs \cite{Main}. In \cite[Theorem 4.5]{DsouzaUma2}, Dsouza gives a necessary and sufficient condition on the associated digraph for a given real Bott manifold to have a spin structure. In \cite{GuclukanIlhanGurbuzer}, Güçlükan İlhan and Gürbüzer show that for every generalized Bott manifold $B_k$, there is an associated acyclic digraph $D_{B_k}$ on labeled vertices $\{v_1,\cdots,v_k\}$ where each edge from a vertex $v_i$ has a vector weight in $\mathbb{Z}_2^{n_i}$. In Section \ref{sect:digraphpart}, we generalize the condition given by Dsouza and Uma to a condition on $D_{B_k}$ for the associated generalized Bott tower $B_k$ to have a Spin structure (Theorem \ref{thm:SpinDigraph}).

The Wu formula implies that $w_3(M)=0$ whenever $w_1(M)$ and $w_2(M)$ are zero. Therefore, the result of Section \ref{sect:SpinStr} give us a sufficient conditions for $w_3(M)$ to be zero. In Section \ref{sect:higherSW}, we obtain a formula for $w_3(M)$ when $M$ is a small cover over a product of simplices of dimensions greater than equal to $3$. As a corollary, we give a necessary conditions for the vanishing of the third Stiefel-Whitney class of $M$. We obtain similar results for $w_4$ and we classify small covers over a product of simplices of dimensions greater than equal to $4$ whose first four Stiefel-Whitney classes are zero. 

\section{Small covers}	
\label{sect:Prelim}
Let $P$ be an $n$-dimensional simple convex polytope and $\mathcal{F}(P)=\{F_1,F_2,\dots, F_m\}$ be the set of facets of $P$. A small cover over $P$ is an $n$-dimensional smooth closed manifold M with a locally standard $\mathbb{Z}_2^n$-action whose orbit space is $P$. 
Two small covers $M_1$ and $M_2$ over $P$ are said to be Davis-Jansukiewicz equivalent if there is a weakly $\Z$-equivariant homeomorphism between $M_1$ and $M_2$ covering the identity on $P$. The Davis-Januskiewicz classes of small covers over $P$ are given by the characteristic functions. 

A characteristic function $\lambda: \mathcal{F}(P) \rightarrow \mathbb{Z}_2^n$ over $P$ is a $\mathbb{Z}_2^n$-coloring function satisfying the following non-singularity condition: 
$$F_{i_1} \cap \cdots \cap F_{i_n}\neq \emptyset \ \ \ \Rightarrow \\ \langle \lambda(F_{i_1}), \dots, \lambda(F_{i_n})\rangle=\Z.$$ In \cite{Main}, Davis and Janueskiewicz construct a small cover $M(\lambda)$ associated to a given characteristic function $\lambda$ as the quotient space of the space $(P \times \Z)$ and the equivalence relation defined by $$(p,g)\sim(q,h) \ \mathrm{if} \ p=q \ \mathrm{and} \ g^{-1}h \in \langle\lambda(F_{i_1}), \dots, \lambda(F_{i_k})\rangle$$ where the intersection $ \overset{k}{\underset{j=1}{\bigcap}} F_{i_j}$ is the minimal face containing $p$ in its relative interior.

\begin{theorem}\label{thm:DJ-class}\cite[Proposition 1.8]{Main} For every small cover $M$ over $P$, there is a characteristic function $\lambda$ with $\Z$-homeomorphism $M(\lambda) \rightarrow M$ covering the identity on $P$.
\end{theorem}

The group $GL(n, \mathbb{Z}_2)$ acts freely on the set of characteristic functions over $P$ by composition. Moreover, the orbit space of this action is in one-to-one correspondence with the Davis-Januskiewicz equivalence classes of small covers over $P$. Fix a basis  $e_1,\cdots, e_n$ for $\Z$ and reorder facets of $P$ in such a way that $\underset{i=1}{\overset{n}{\bigcap}}F_i \neq \emptyset$. By the above theorem, for a given small cover $M$ over $P$, there is an $(n \times (m-n))$-matrix $A=[a_{ij}]$ such that $M$ and $M(\lambda)$ are Davis-Januszkiewicz equivalent where $$ \lambda(F_i)= \begin{cases} 
e_i, & i\leq n \\
\displaystyle \sum_j a_{ji}e_j & i>n. 
\end{cases}$$ 

\begin{theorem}[Theorem 4.14, \cite{Main}] The mod $2$ cohomology ring of $M$ is $\mathbb{Z}[P]/J$, where $J$ is the homogeneous ideal generated by the monomials $\displaystyle x_i+\sum_{j=1}^{m-n}a_{ij}x_{n+j}$. 
\end{theorem}

Let $w_i(M)$ and $w(M)$ denote the $i$-th and the total Stiefel-Whitney classes of $M$, respectively. By Corollary 6.8 in \cite{Main}, the total-Stiefel Whitney class of a small cover over $M$ is given by the equation (\ref{eq:total}).   Let $A_j$ denote the $j$-th column vector of $A$. 
Then the first Stiefel-Whitney class of $M$ is given by the following formula
$$w_1(M)=\sum_{i=1}^{m-n} (1+\sum_ja_{ji})\cdot x_{i+n}= \sum_{i=1}^{m-n} (1+ A_i\cdot A_i) \cdot x_{i+n}$$ since $a_{ji}^2=a_{ji}$. Hence $M$ is orientable if and only if $A_i\cdot A_i\equiv 1 \ (\mathrm{mod} \ 2)$ for all $1\leq i \leq m-n$. By comparing the degree $2$-terms in each side of the equation (\ref{eq:total}), one obtains a similar formula for the second Stiefel-Whitney class of $M$.

\begin{proposition} \label{prop:secondSW} The second Stiefel-Whitney class of $M$ is 
	\begin{eqnarray} \label{eq:seconSW} w_2(M)= \sum_{i=1}^{m-n} \alpha_i \cdot x_{i+n}^2 +\sum_{1\leq i<j\leq m-n} \beta_{ij}\cdot x_{i+n} \cdot x_{j+n} \ \quad \qquad \ (\mathrm{mod} \ I)\end{eqnarray} where $\displaystyle \alpha_i=\binom{1+A_i\cdot A_i}{2}$ and 
	$\beta_{ij}=(1+A_i\cdot A_i)(1+A_j\cdot A_j)+A_i\cdot A_j$.
\end{proposition}
\begin{proof} The coefficient of $x_{i+n}^2$ in the equation (\ref{eq:total}) equals to the coefficient of $y^2$ in $\displaystyle (1+y)\Big(\underset{j}{\prod}(1+a_{ji}y)\Big)$ which is the $(k_i+1)$-th power of $1+y$ where $k_i$ is the number of $1$'s in $A_i$. Since the entries of $A_i$ are either $0$ or $1$, the number of $1$'s in $A_i$ is equal to $A_i \cdot A_i$. Hence the coefficient of $x_{i+n}^2$ in (\ref{eq:total}) is
$\binom{1+A_i \cdot A_i}{2}.$

To find $\beta_{ij}$, first note that $|\{t| \ a_{ti}=a_{tj}=1\}|=A_i \cdot A_j$. Therefore $\beta_{ij}$ is equal to the coefficient of $y_iy_j$ in the product $$(1+y_i)^{(A_i\cdot A_i-A_i\cdot  A_j+1)} (1+y_j)^{(A_j\cdot A_j-A_i\cdot  A_j+1)}(1+y_i+y_j)^{A_{ij}}.$$ Hence we have 
	\begin{eqnarray*}\beta_{ij}&=&(A_i\cdot A_i-A_i\cdot  A_j+1)(A_j \cdot A_j+1)+A_{ij}(A_j \cdot A_j)\\
	&=&	(1+A_i\cdot A_i)(1+A_j\cdot A_j)-A_i\cdot A_j. \end{eqnarray*}
Since we work with $\mathbb{F}_2$ coefficients, the result follows.
\end{proof}
\begin{corollary}\label{cor:spingeneral}
	Let $M$ be a small cover over $P$ with an associated reduced matrix $A$. If $A_i \cdot A_i\equiv 3 \ (\mathrm{mod} \ 4)$ and $A_i \cdot A_j\equiv 0 \ (\mathrm{mod} \ 2)$ for all  possible $ i < j $ then $M$ has a Spin structure.
\end{corollary}

\section{Existence of Spin Structure}
\label{sect:SpinStr}

In this section, we give a necessary and sufficient condition for the existence of Spin structure for generalized Bott manifolds. Let $B_k$ be a generalized real Bott manifold given in
(\ref{eq:BGott}). One can realize $B_k$ as a small cover over $P=\underset{i=1}{\overset{k}{\prod}} \Delta^{n_i}$ where $\underset{i=1}{\overset{k}{\sum}} n_i=n$. The facets of $P$ is given by the following set
$$\mathcal{F}=\{F^i_j=\Delta^{n_1} \times \cdots \times \Delta^{n_{i-1}}\times f^i_{j}\times \Delta^{n_{i+1}}\times \cdots \times \Delta^{n_k} | \ 1 \leq i \leq k, \ 0 \leq j \leq n_i \}$$ where $\{f_{0}^i,\dots, f^i_{n_i}\}$ is the set of facets of the simplex $\Delta^{n_i}$. Note that $P$ has $(n+k)$-facets and the intersection $\displaystyle \bigcap_{j \neq 0} F_{j}^i$ is non-empty. Hence $B_k$ can be represented by a $(n\times k)$ matrix $A=[a_{ij}]$ by choosing $F_l=F^i_{j}$ for $l=n_1+\cdots+n_{i-1}+j$ and $1 \leq j \leq n_i$ and $F_l=F^i_0$ for $l=n+i$. Following \cite{Choi, Main}, one can see $A$ as a $(k\times k)$ vector matrix $A=[\mathbf{v}_{ij}]$ where $\mathbf{v}_{ij}\in \mathbb{Z}_2^{n_i}$. Here $\mathbf{v}_{ij}$ is the column vector whose $l$-th entry is $a_{n_1+\cdots+n_{i-1}+l,j}$.

Note that facets in $\mathcal{F} \backslash \{F^1_{j_1},\cdots,F^k_{j_k}\}$ intersect at a vertex for every $0\leq j_i\leq n_i$ and $1\leq i \leq k$. Moreover, a family of facets containing the set $\{F^i_{0},\cdots, F^i_{n_i}\}$ has an empty intersection for any $1\leq i \leq k$. Let $A_{l_1\cdots l_k}$ be a $(k\times k)$ matrix whose $j$-th row is the $l_j$-th row of $A$ for $1\leq l_i \leq n_i$ and $1\leq i \leq k$. In \cite{Main}, using these facts, it is shown that the characteristic function corresponding to $A$ satisfies the non-singularity condition if and only if every principal minor of $A_{l_1\cdots l_k}$ is $1$ for all $1\leq l_i \leq n_i$ and $1\leq i \leq k$. This forces $(\mathbf{v}_{ii})_t=1$ for all $1\leq i \leq k$ and $1\leq t\leq n_i$.

Note that the Stanley-Reisner ring of $P$ is $$ \mathbb{Z}_2[x_{10},\cdots,x_{1n_{1}}, \cdots, x_{k0},\cdots,x_{kn_{k}}]/ I$$ where $I$ is the homogeneous ideal generated by monomial products $x_{i0}\cdots x_{in_{i}}$, $1\leq i \leq k.$ In this notation, $x_{ij}$ corresponds to $x_{n_1+\cdots+n_{i-1}+j}$ when $1\leq j \leq n_{i}$ and to $x_{n+i}$ when $j=0$ in the equation (\ref{eq:total}). Therefore the second Stiefel-Whitney class of $B_k$ is equal to
\begin{eqnarray*}  w_2(M)= \sum_{i=1}^{k} \alpha_i \cdot x_{i0}^2 +\sum_{1\leq i<j\leq k} \beta_{ij}\cdot x_{i0} \cdot x_{j0}\end{eqnarray*} modulo $I$ where $\alpha_i$ and 
$\beta_{ij}$ are as given in Proposition \ref{prop:secondSW}. From now on, we assume that $n_i=1$ for $1\leq i\leq l$ and $n_i>1$, otherwise.  This means that the only relations involving the monomials of degree $2$ are
$$x_{i0}^2=\underset{j\neq i}{\sum} \mathbf{v}_{ij}\cdot x_{i0}\cdot x_{j0}$$ for $1\leq i \leq l$ (here, the vector $\mathbf{v}_{ij} \in \mathbb{Z}_2$ is considered as a scalar) . Therefore we have
\begin{eqnarray}\label{eq:SSWforPS} w_2(M)&=&\sum_{i=l+1}^{k} \alpha_i \cdot x_{i0}^2 
+\sum_{l \leq i<j\leq k } \beta_{ij}\cdot x_{i0} \cdot x_{j0}\\&&+ \sum_{i <j \leq l} (\beta_{ij}+\mathbf{v}_{ij} \cdot \alpha_{i}+\mathbf{v}_{ji} \cdot \alpha_{j}) \cdot x_{i0}\cdot x_{j0} \nonumber \\
&&+\sum_{   i < l+1 \leq j \leq k}(\beta_{ij}+\mathbf{v}_{ij} \cdot \alpha_i) \cdot x_{i0} \cdot x_{j0} \nonumber
\end{eqnarray}

\begin{theorem}\label{theorem:seconSWforPS}
	The generalized real Bott manifold $B_k$ has a Spin structure if and only if the following conditions are satisfied:
	\begin{itemize}
		\item[i)] $A_i \cdot A_i \equiv 1 \ (\mathrm{mod} \ 2)$ when $ i\leq l$ and $A_i \cdot A_i\equiv 3 \ (\mathrm{mod} \ 4)$, otherwise, 
		\item[ii)] $A_i \cdot A_j\equiv 0 \ (\mathrm{mod} \ 2)$ for all $l \leq i < j \leq k$,
		\item[iii)] $A_i\cdot A_j$ and 
		$\dfrac{\mathbf{v}_{ij}\cdot (A_i\cdot A_i+1)+\mathbf{v}_{ji} \cdot (A_j\cdot A_j+1)}{2}$ have the same parity when $1\leq i < j \leq l.$ \\
		\item[iv)] $A_i\cdot A_j$ and $ \dfrac{\mathbf{v}_{ij}\cdot (A_i\cdot A_i+1)}{2}$ have the same parity when $1 \leq i <l+1\leq j \leq k.$ 
	\end{itemize}  
\end{theorem} 
\begin{proof} The manifold $B_k$ has a Spin structure if and only if it is orientable and $w_2$ vanishes. Recall that the manifold $B_k$ is orientable if and only if $A_i \cdot A_i$ is congruent to $1$ modulo $2$. In this case, $\beta_{ij} \equiv A_i \cdot A_j$ modulo $2$. Then the theorem follows from the equation (\ref{eq:SSWforPS}) and the fact that $\binom{1+A_j \cdot A_j}{2}$ have the same parity with $\frac{A_i \cdot A_j +1}{2}$.
\end{proof}
It is well-known that the vector matrix $A$ is equivalent to an upper triangular one in which the entries of the diagonal vectors are all 1 via conjugation by a permutation matrix \cite[Lemma 5.1]{Main}. Under this assumption, the above theorem is equivalent to the \cite[Theorem 4.7]{DsouzaUma}. In \cite[Theorem 4.7]{DsouzaUma}, the ordering in the product is chosen so that the last $k-l$ of the simplices have dimension $1$. Moreover, $T_s$ and $T_{rs}$ in \cite[Theorem 4.7]{DsouzaUma} are equivalent to $\binom{A_s\cdot A_s}{2}$ and $A_r \cdot A_s$, respectively and the orientability condition is equivalent to $A_s\cdot A_s \equiv 1 \ (\mathrm{mod} \ 2)$. 


By Theorem \ref{theorem:seconSWforPS}, it follows that the converse of Corollary \ref{cor:spingeneral} is also true when $l=0$.
\begin{corollary}\label{cor:seconSWforPS}
	The generalized real Bott manifold with $l=0$ has a Spin structure if and only if $A_i \cdot A_i\equiv 3 \ (\mathrm{mod} \ 4)$ and $A_i \cdot A_j\equiv 0 \ (\mathrm{mod} \ 2)$ for all $1\leq i < j \leq k$ where $A$ is the reduced matrix.   
\end{corollary} 

\begin{example}
Let $P=\Delta^2\times \Delta^3\times \Delta^5$ and $B$ be a $3$-step generalized Bott manifold corresponding to the reduced matrix 
\[A=
\left[ 
\begin{array}{c|c|c} 
	1 & 0 & 0 \\
	1 & 0 & 0 \\
	\hline
	0 & 1 & 1\\
	1 & 1 &1\\
	1 & 1 &0\\
	\hline
	1 & 0 &1\\
	1 & 0 &1\\
	1 & 0 &1\\
	0 & 0 &1\\
	0 & 0 &1
\end{array} 
\right].  \] Then $B$ has a Spin structure by the above Corollary.
\end{example}
The following corollary also follows from the Proposition 5.1 of \cite{Shen}.
\begin{corollary} If a generalized real Bott manifold over $P=\underset{t=1}{\overset{k}{\prod}} \Delta^{n_t}$ with $l=0$ admits a Spin structure then $n_j \equiv 3$ (mod $4$) for some $j$.
	\begin{proof} Let $P=\underset{t=1}{\overset{k}{\prod}} \Delta^{n_t}$ and $M$ be a small cover over $P$ with an associated vector matrix $A$. If $B$ is a vector matrix obtained by conjugating $A$ via permutation matrix $P_{\sigma}$ then $A_i\cdot A_i=B_{\sigma(i)} \cdot B_{\sigma(i)}$ and $A_i \cdot A_j= B_{\sigma(i)}\cdot B_{\sigma(j)}$. Therefore, we can assume that $A$ is an upper triangular vector matrix in which the entries of the diagonal vectors are all 1. Then we have $A_1\cdot A_1=n_1$. So if $M$ has a Spin structure, $n_1 \equiv 3$ (mod $4$). 		
	\end{proof}
\end{corollary}
When $l=k$, we have the following result.
\begin{corollary}\label{cor:1dimensionals}
	The real Bott manifold $B_k$ has a Spin structure if and only if 
	\begin{itemize}
		\item[i)] $A_i \cdot A_i \equiv 1 \ (\mathrm{mod} \ 2)$, $1\leq i \leq k$,
		\item[ii)] $A_i\cdot A_j$ and $\dfrac{\mathbf{v}_{ij}\cdot (A_i\cdot A_i+1)+\mathbf{v}_{ji} \cdot (A_j\cdot A_j+1)}{2}$ have the same parity when $1 \leq i < j \leq l.$ \\
	\end{itemize}
\end{corollary}

The above corollary is equivalent to  Theorem 3.2 in \cite{DsouzaUma2} where $A$ is assumed to be upper-triangular. In particular, Theorem 1.2 in \cite{Gasior1} directly follows from the corollary. 
\begin{example}
	Let $P=I \times \Delta^2 \times \Delta^2$. Then a small cover $B_3$ over $P$ corresponds to a vector  matrix 
	\[A=
	\left[ 
	\begin{array}{c|c|c} 
	1 & a_{12} &a_{13}\\
	\hline
	a_{21} & 1 &a_{23}\\
	a_{31} & 1 &a_{33}\\
	\hline
	a_{41} & a_{42} &1\\
	a_{51} & a_{52} &1
	\end{array} 
	\right].  \]
	If $B_3$ has a Spin structure then $a_{12}+a_{42}+a_{52}=1$ and $a_{13}+a_{23}+a_{33}=1$ by part i) of Theorem \ref{theorem:seconSWforPS} and $a_{12}a_{13}+a_{23}+a_{33}+a_{42}+a_{52}\equiv 0$ (mod $2$) by part ii) of Theorem \ref{theorem:seconSWforPS}. By substituting the first two equations to the last one, we get $a_{12}=a_{13}=0$ and hence $a_{42}+a_{52}=a_{23}+a_{33}=1$. On the other hand, at least one of the vectors $\displaystyle \begin{pmatrix}
	a_{42}  \\ 	a_{52}	\end{pmatrix}$ and $\begin{pmatrix}	a_{23}  \\ 	a_{33}
	\end{pmatrix}$ must be zero by the non-singularity condition. Hence there is no small cover over $I \times \Delta^2 \times \Delta^2$ with a Spin structure when $n\geq 2$.
\end{example}

It is well-known that when $n_i$'s are all even, there is no orientable small cover over $P$ \cite{AltunbulakGuclukanIlhan}. Hence small covers over $P$ have no Spin structures when all the $n_i$'s are even. In the next section, we generalize the above example to have a non-existence result for every small cover over $P=I\times \Delta^{2n_1} \times \cdots \times \Delta^{2n_k}$ for $k \geq 2$. When $k=1$, a small cover over $I\times \Delta^{4t}$ does not have a Spin structure since $A_2 \cdot A_2$ is either $4t$ or $4t+1$. However, the small cover over $P=I\times \Delta^{4t+2}$ corresponding to a characteristic function $\lambda$ which sends $F^1_0$ to $e_1$ and $F^2_0$ to $e_1+e_2+\cdots+e_{4t+3}$ has a Spin structure.

Given a  dimension function $\omega: \{1,2,\dots,n\} \rightarrow \mathbb{N}$, let $I_{\omega}$ be the identity vector matrix associated to $\omega$, i.e, the $(i,j)$-entry of $I_{\omega}$ is $1$ when $\omega(1)+\cdots+\omega(j-1)+1\leq i \leq \omega(1)+\cdots+\omega(j)$, and $0$, otherwise. 
To generalize Theorem 1.2 in \cite{Gasior1} to our case,  we  denote the matrix $A-I_{\omega}$ where $\omega(i)=n_i$ by $B$. 
\begin{proposition}\label{prop:Gasior1} The generalized real Bott manifold with an associated matrix $B$ has a Spin structure if and only if for all $1 \leq i < j \leq k$, the generalized Bott manifold corresponding to $B_{ij}$ has a Spin structure, where $B_{ij}$ is the vector matrix whose $l$-th column is $B_l$ if $l=i,j$ and $0$, otherwise.   
\end{proposition}

\section{$\omega$-weighted digraph interpretation}
\label{sect:digraphpart}
In \cite{Choi}, Choi shows that there is a bijection between the set of real Bott manifolds and  acyclic digraphs with $n$-labeled vertices which sends $B_k$ to a graph whose adjacency matrix is $A-I_k$. In \cite[Theorem 4.5]{DsouzaUma2}, Dsouza and Uma give an interpretation of existence of a Spin structure for real Bott manifolds in terms of associated digraphs.  In this section, we generalize \cite[Theorem 4.5]{DsouzaUma2} to small covers over a product of simplices.

\begin{definition}  Given a dimension function $\omega: V \rightarrow \mathbb{N}$, a digraph with vertex set $V$ is called $\omega$-vector weighted if every edge $(u,v)$ is assigned a non-zero vector $\mathbf{w(u,v)}$ in $\mathbb{Z}_2^{\omega(u)}$. 
\end{definition}

Let $G$ be a $\omega$-vector weighted digraph. For convenience, we take the weight of $(u,v)$ to be the zero vector in $\mathbb{Z}_2^{\omega(u)}$ when there is no edge from $u$ to $v$. If $(u,v)$ is an edge of $G$ then $u$ is called an in-neighbor of $v$ and $v$ is called an out-neighbor of $u$. Let $N^{-}_G(v)$ and $N^{+}_G(v)$ denote the set of in-neighbors and out-neighbors of $v$ in $G$. We define in-degree $\mathrm{deg}^{-}(v)$ and out degree $\mathrm{deg}^{+}(v)$ of $v$ as follows:
\begin{eqnarray*} \mathrm{deg}^{-}(v)&=&\sum_{u \in N^{-}_G(v)} \mathbf{w(u,v)} \cdot \mathbf{w(u,v)} \\ \mathrm{deg}^{+}(v)&=&\sum_{z \in N^{+}_G(v)} \mathbf{w(v,z)} \cdot \mathbf{w(v,z)}.\end{eqnarray*}

We can consider a digraph as an $\omega$-weighted digraph with $\omega(i)=1$ for each $i$. In this case, the notion of in-degree and out-degree of a vertex of a $\omega$-weighted digraph agrees with those of digraphs. An adjacency matrix $A_{\omega}(G)$ of an $\omega$-weighted digraph $G$ with labeled vertices $v_1,\cdots, v_n$ is defined to be an $(n \times n)$ $\omega$-vector matrix whose $(i,j)$-th entry is $\mathbf{w(v_i,v_j)}$. An $\omega$-vector weighted digraph is called acyclic if it does not contain any directed cycle.  

As shown in \cite{GuclukanIlhanGurbuzer}, there is a one to one correspondence between the set of small covers over the product $P=\Delta^{n_1}\times \cdots \Delta^{n_k}$ and the set of acylic $\omega$-weighted digraphs where $\omega:\{v_1,\cdots,v_k\} \to \mathbb{N}$ is defined by $\omega(v_i)=n_i$. The correspondence is obtained by sending a small cover with an associated matrix $A$ to a $\omega$-weighted digraph whose adjacency matrix is $A-I_{\omega}.$ For a given small cover $B$ over $P$, we denote the associated acyclic $\omega$-weighted digraph by $D_B$. Recall that the dot product of a vector $\mathbf{v}$ over $\mathbb{Z}_2$ with itself is equal to the number of non-zero coordinates of $\mathbf{v}$. Therefore  $A_i \cdot A_j$ is equal to $A_{\omega}(D_B)_i \cdot A_{\omega}(D_B)_i+\omega(i)$ when $i=j$ and $A_{\omega}(D_B)_i \cdot (A_{\omega}D_B)_j+ \mathbf{w(v_i,v_j)}\cdot \mathbf{w(v_i,v_j)}+\mathbf{w(v_j,v_i)}\cdot \mathbf{w(v_j,v_i)},$ otherwise. Moreover $A_{\omega}(D_B)_i \cdot A_{\omega}(D_B)_i$ is equal to $ \mathrm{deg}^{-}(v_i)$. Let $M_{ij}$ be the sum of $\omega(u,v_i) \cdot \omega(u,v_j)$ where $u$ runs in the set of in-neighbor of both $v_i$ and $v_j$. Then $A_{\omega}(D_B)_i\cdot A_{\omega}(D_B)_j=M_{ij}$. 

\begin{theorem}  \label{thm:SpinDigraph}
	The generalized real Bott manifold $B$ with associated $w$-weighted digraph $D_B$ has a Spin structure if and only if the following conditions are satisfied:
	\begin{itemize}
		\item[i)] Indegree of a vertex $v$ of $D_B$ is even if $\omega(v)=1$ and is congruent to $-\omega(v)+3$ modulo $4$, otherwise,
		\item[ii)] $M_{ij}$ is even if $v_i$ is neither in-neighbour nor out-neighbour $v_j$ with $i \neq j$,
		\item[iii)] $M_{ij}$ and $\displaystyle \dfrac{\mathbf{w(v_i,v_j)} \cdot \mathrm{deg}^{-}(v_i)}{2}$ have the same parity when $v_i$ is an in-neighbor of $v_j$ with $\omega(v_i)=1$, \\
		\item[iv)] $M_{ij}$ and $\displaystyle \mathbf{w(v_i,v_j) }\cdot \mathbf{w(v_i,v_j) }$ have the same parity when $v_i$ is an in-neighbor of $v_j$ with $\omega(v_i)>1 $.
	\end{itemize}
\end{theorem}
\begin{proof}
	If $v_i$ is neither in-neighbour nor out-neighbour $v_j$, the condition iii) and iv) of Theorem \ref{theorem:seconSWforPS} is equivalent to the statement that $A_i \cdot A_j$ is even for $i \neq j$.  In this case, we also have $M_{ij}=A_i \cdot A_j$. Otherwise, either $v_i$ or $v_j$ is an in-neighbour of the other one. Since $M_{ij}=M_{ji}$, without loss of generality, we can assume that $v_i$ is. Then $A_i\cdot A_j= M_{ij}+ \mathbf{w(v_i,v_j)}\cdot \mathbf{w(v_i,v_j)}$. Therefore, when $\omega(v_i)=1$, combaining the conditions iii) and iv) of Theorem \ref{theorem:seconSWforPS}, one obtains condition iii) above. When $\omega(v_i)>1$, iv) can be obtained by combaining part ii) and iv) of Theorem \ref{theorem:seconSWforPS}.   
	
\end{proof}

\begin{example} Let $P=\Delta^2\times \Delta^3\times \Delta^3\times \Delta^3$ and $B$ be a $4$-step generalized Bott manifold corresponding to the reduced matrix 
	\[A=
	\left[ 
	\begin{array}{c|c|c|c} 
	1 & 1 & 0 & 1\\
	1 & 0 & 0 & 1\\
	\hline
	0 & 1 & 0&0\\
	0 & 1 &0&0\\
	0 & 1 &0&0\\
	\hline
	0 & 0 &1&0\\
	0 & 0 &1&0\\
	0 & 0 &1&0\\
	\hline
	0 & 1 &1&1\\
	0 & 1 &0&1\\
	0 & 1 &1&1
	\end{array} 
	\right].  \]
	Then $\omega:\{1,2,3,4\} \rightarrow \mathbb{N}$ with $\omega(1)=2$, $\omega(2)=\omega(3)=\omega(4)=3$ and an $\omega$-weighted digraph corresponding to $B$ is as given below.
	\begin{figure}[h]
		\centering
		\begin{tikzpicture}[
		> = stealth, 
		shorten > = 1pt, 
		auto,
		node distance = 3cm, 
		semithick 
		]
		
		\tikzstyle{every state}=[
		draw = black,
		thick,
		fill = white,
		minimum size = 3mm
		]
		
		\node[state] (v1) {$v_1$};
		\node[state] (v2) [above right of=v1] {$v_2$};
		\node (v) [right of=v1]{} ;
		\node[state] (v3) [above right of=v] {$v_3$};
		\node[state] (v4) [right of=v] {$v_4$};
		\node(a) at (3,-1) {$G$} ;
		
		\path[->] (v1) edge node {$\big(\begin{smallmatrix}
			1 \\
			0 
			\end{smallmatrix}\big)$} (v2);

		\path[->] (v4) edge node [near start, swap] {$\Big(\begin{smallmatrix}
			1 \\
			0 \\
			1
			\end{smallmatrix}\Big)$} (v3);
		\path[->] (v1) edge node [swap] {$\big(\begin{smallmatrix}
			1 \\
			1 
			\end{smallmatrix}\big)$} (v4); 
		\path[->] (v4) edge node [above, pos=0.5] {$\Big(\begin{smallmatrix}
			1 \\
			1\\
			1 
			\end{smallmatrix}\Big)$} (v2);  
		\end{tikzpicture}
	\end{figure}
	Since $\mathrm{deg}^{-}(v_3)=2$, $B$ has no Spin structure by part i) of the above theorem.
\end{example}

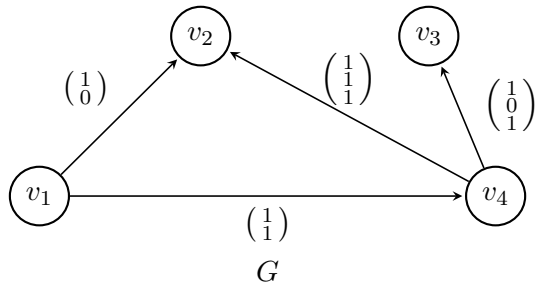
\begin{corollary}
	A small cover over $P=I \times \Delta^{2n_1}\cdots \times \Delta^{2n_k}$ does not have a Spin structure when $k\geq 2$.  
\end{corollary}

\begin{proof}
	Let $M$ be a small cover over $P$ and $G$ be the associated acyclic $\omega$-weighted digraph. Assume for a contradiction that $M$ has a Spin structure. The underlying digraph of $G$ has a source, say $v_i$. Since indegree of $v_i$ is zero, the weight of $v_i$ must be $1$. Let $v_j$ be a source of the digraph obtained by removing $v_i$ from the underlying digraph and, $v_k$ be a source of the digraph obtained by removing $v_i$ and $v_j$. Then the in-degrees of vertices $v_j$ and $v_k$ are $ \mathbf{w(v_i,v_j)}$ and $\mathbf{w(v_i,v_k)}+\mathbf{w(v_j,v_k)}\cdot \mathbf{w(v_j,v_k)}$, respectively. By part i) of the above theorem, both of them must be odd. In particular $\mathbf{w(v_i,v_j)}=1$ and, $\mathbf{w(v_i,v_k)}$ and $\mathbf{w(v_j,v_k)}\cdot \mathbf{w(v_j,v_k)}$ have different parities. On the other hand,
	$M_{jk}=\mathbf{w(v_i,v_k)}$ as a dot product of $j$-th and $k$-th column of the adjacency matrix. Since $M_{jk}$ and $\mathbf{w(v_j,v_k)}\cdot \mathbf{w(v_j,v_k)}$ have different parities, $v_j$ can not be an in-neighbor of $v_k$, by part iv) of the above theorem. This means that $\mathbf{w(v_j,v_k)}$ is the zero vector. Hence by part ii), $M_{jk}$ must be even and hence $\mathbf{w(v_i,v_k)}=0$. Contradiction.  
\end{proof}
\section{Higher Stiefel-Whitney Classes}
\label{sect:higherSW}

It is well-known that the Stiefel-Whitney classes $w_i$ of a smooth manifold satisfy the Wu formula \cite{May}
$$ Sq^i(w_j)=\sum_{t=0}^i  \binom{j+t-i-1}{t} w_{i-t}w_{j+t}$$
where $Sq^i$ denotes the Steenrod squares. Therefore, for any $i \leq j$ with $i+j=m$, one has
$$\binom{j-1}{i}w_m=Sq^i(w_m)+\sum_{t=0}^{i-1}  \binom{j+t-i-1}{t} w_{i-t}w_{j+t}.$$
Substituting $m=3$ and $i=1$ gives $w_3=Sq^1(w_2)+w_1w_2$. This means that whenever $w_1$ and $w_2$ are both zero, so is $w_3$. Therefore the following result directly follows from Corollary \ref{cor:seconSWforPS}.

\begin{proposition}\label{prop:firstthree} The first three Stiefel-Whitney classes of a small cover over a product of simplices of dimensions greater than equal to $2$ are zero if and only if $A_i\cdot A_i \equiv 3$ (mod $4$) and $A_i \cdot A_j \equiv 0$ (mod $2$) for all $i \neq j$ where $A$ is the associated reduced matrix.
\end{proposition}

Now we show that the conditions of the above proposition are not necessary for $w_3(M)$ to be zero. For this, let $k_{S}(A)$ denote the size of the set $\{t| \ a_{ts}=1 \ \mathrm{for \ all} \ s \in S\}$ for any $S \subseteq \{1,2,\cdots, k\}$ . We write $k_S$ instead of $k_S(A)$ when it is clear from the context. Note that $k_{\{i\}}=A_i \cdot A_i$ and $k_{\{i,j\}}=A_i \cdot A_j$.

\begin{theorem} The third Stiefel-Whitney class of a small cover $M$ over $P=\underset{i=1}{\overset{k}{\prod}} \Delta^{n_i}$ modulo $I$ is equal to
	$$w_3(M)=\sum_{1\leq i\leq k} \binom{k_{\{i\}}+1}{3}x_{i0}^3 +\sum_{i\neq j} P(i,j) x_{i0}^2x_{j0}+ \sum_{i_1<i_2<i_3} Q(i_1,i_2,i_3)x_{i_10}x_{i_20}x_{i_30}$$ where 
\begin{eqnarray}\label{eq:thirdSW} P(i,j)&=&\binom{k_{\{i\}}+1}{2}\cdot \big(k_{\{j\}}+1\big)-k_{\{i\}}\cdot k_{\{i,j\}},\\
Q(i_2,i_2,i_3)&=&\Big(\prod_{p=1}^3\big(k_{\{i_p\}}+1\big)\Big)+\sum_{p=1}^3(k_{\{i_p\}}+1)\cdot k_{\{i_1,i_2,i_3\} -\{i_p\}}.
\end{eqnarray}
\end{theorem}
\begin{proof} One can easily find the coefficient of $x_{i0}^3$ as in the Stiefel-Whitney classes of smaller dimensions. The coefficient of $x_{i0}^2x_{j0}$ is equal to the coefficient of $y_1^2y_2$ in the polynomial $$(1+y_1)^{k_{\{i\}}-k_{\{i,j\}}+1}(1+y_2)^{k_{\{j\}}-k_{\{i,j\}}+1}(1+y_1+y_2)^{k_{\{i,j\}}}$$ as before. We can pick $y_2$ either from the factor $(1+y_2)^{k_{\{j\}}-k_{\{i,j\}}+1}$ or from the factor $(1+y_1+y_2)^{k_{\{i,j\}}}$. If we chose it from the second one, we have to choose $y_1^2$ from $(1+y_1)^{k_{\{i\}}-k_{\{i,j\}}+1}(1+y_1+y_2)^{k_{\{i,j\}}-1}$. Therefore, we have
	\begin{eqnarray*}P(i,j)&=&(k_{\{j\}}-k_{\{i,j\}}+1)\binom{k_{\{i\}}+1}{2}+k_{\{i,j\}}\binom{k_{\{i\}}}{2} \\
	&=&\binom{k_{\{i\}}+1}{2}(k_{\{j\}}+1)-k_{\{i,j\}}.\end{eqnarray*}

The coefficient of the monomial $x_{i_10}x_{i_20}x_{i_30}$ in $w_3(M)$ is equal to the coefficient of $y_1y_2y_3$ in the product
\begin{eqnarray}\label{eq:coeffQ}\Bigg( \prod_{j=1}^3 (1+y_j)^{k_{\{i_j\}}-\underset{p\neq j}{\sum}k_{\{i_p,i_j\}}+k_{\{i_1,i_2,i_3\}}+1} \Bigg) \cdot \Bigg (\prod_{p\neq q}(1+y_p+y_q)^{k_{\{i_p,i_q\}}-k_{\{i_1,i_2,i_3\}}} \Bigg) \cdot \Big( 1+y_1+y_2+y_3\Big)^{k_{\{i_1,i_2,i_3\}}} \end{eqnarray} 
Now we can choose $y_3$ from either of the factors $(1+y_3)^{k_{\{i_3\}}-\underset{p\neq 3}{\sum}k_{\{i_p,i_3\}}+k_{\{i_1,i_2,i_3\}}+1}$, $(1+y_1+y_3)^{k_{\{i_1,i_3\}}-k_{\{i_1,i_2,i_3\}}}$, $(1+y_2+y_3)^{k_{\{i_2,i_3\}}-k_{\{i_1,i_2,i_3\}}}$ or $( 1+y_1+y_2+y_3\Big)^{k_{\{i_1,i_2,i_3\}}}$. Therefore we have
\begin{eqnarray*}Q(i_1,i_2,i_3)&=&\Big(k_{\{i_3\}}+k_{\{i_1,i_2,i_3\}}+1-\underset{p\neq 3}{\sum}k_{\{i_p,i_3\}} \Big) \cdot\Big(\big(1+k_{\{i_1\}}\big)\big(1+k_{\{i_2\}}\big)+k_{\{i_1,i_2\}}\Big)\\&&+\Big(\sum_{p=1}^2\big(k_{\{i_p,i_3\}}-k_{\{i_1,i_2,i_3\}}\big)\cdot\big(k_{\{i_p\}} \big(1+k_{\{i_1,i_2\}-\{i_p\}}\big)+k_{\{i_1,i_2\}}\big)\Big)\\
&&+k_{\{i_1,i_2,i_3\}}\cdot \big(k_{i_1}k_{i_2}+k_{\{i_1,i_2\}}-1\big).\end{eqnarray*} By algebraically manipulating terms, one can easily obtain the desired formula for $Q(i_1,i_2,i_3)$.
\end{proof}

Note that the above theorem is also true for small covers over an arbitrary simple convex polytope when the cohomology classes are represented appropriately. Moreover, one can easily find a formula for the third Stiefel-Whitney class of a small cover over a product of simplices as in the equation (\ref{eq:SSWforPS}) by taking the relations coming from $I$ into an account. Here, we focus on the case where the dimension of simplies are all greater than equal to $3$ in which $I$ does not contain any relation of dimension $3$ to obtain a simple formula.

\begin{corollary} Let $M$ be a small cover over $P=\underset{i=1}{\overset{k}{\prod}} \Delta^{n_i}$ with $n_i \geq 3$. Then $w_{3}(M)=0$ if and only if the following conditions hold:
\begin{itemize}
	\item[i)] $k_{\{i\}}\not\equiv 2$ (mod $4$),
	\item[ii)]If $k_{\{i\}}$ or $k_{\{j\}}$ is odd then $k_{\{i,j\}}\equiv 1$ (mod $2$) if and only if either $k_{\{i\}}\equiv 0$ (mod $4$) and $k_{\{j\}} \equiv 1$ (mod $4$) or vice a versa,
	\item[iii)] If $k_{\{i_1\}}\equiv k_{\{i_2\}}\equiv k_{\{i_3\}}\equiv 0$ (mod $4$) for $i_1<i_2<i_3$ then $k_{\{i_1,i_2\}}+k_{\{i_1,i_3\}}+k_{\{i_2,i_3\}}\equiv 1$ (mod $2$).
\end{itemize}
 \end{corollary}
\begin{proof} Since $I$ does not contain a monomial of degree less than equal to $3$ when $P=\underset{i=1}{\overset{k}{\prod}} \Delta^{n_i}$ with $n_i \geq 3$, $w_3(M)$ is zero if and only if $\displaystyle \binom{k_{\{i\}}+1}{3} \equiv 0$ (mod $2$) for all $i$, $P(i,j)\equiv 0$ (mod $2$) for all $i\neq j$ and $Q(i_1,i_2,i_3)\equiv 0$ (mod $2$) for all $i_1<i_2<i_3$. Here the first condition is equivalent to the condition $i)$. If neither $k_{\{i\}}$ nor $k_{\{j\}}$ is divisible by $4$ then $P(i,j)\equiv P(j,i)\equiv 0$ (mod $2$) if and only if $k_{\{i,j\}} \equiv 0$ (mod $2$). Let $k_{\{i\}} \equiv 0$ (mod $4$). Then $P(i,j) \equiv 0$ (mod $2$) for all $j \neq i$. Moreover, $P(j,i)\equiv \binom{k_{\{j\}}+1}{2}-k_{\{j\}}k_{\{i,j\}}$ is even if and only if either $k_{\{j\}}\equiv 0$ (mod $4$) or $k_{\{j\}} \equiv 1$ (mod $4$) and $k_{\{i,j\}}\equiv 1$ (mod $2$), or $k_{\{j\}} \equiv 3$ (mod $4$) and $k_{\{i,j\}}\equiv 0$ (mod $2$). Therefore when the condition $i)$ holds, $P(i,j)\equiv P(j,i)\equiv 0$ (mod $2$) if and only if $M$ satisfies the condition $ii)$.
	
Now suppose that the conditions $i)$ and $ii)$ are hold. If $k_{\{i_1\}}, k_{\{i_2\}} $ and $k_{\{i_3\}}$ are all divisible by $4$ then we have 
$$Q(i_1,i_2,i_3)\equiv 1+k_{\{i_2,i_3\}}+k_{\{i_1,i_3\}}+k_{\{i_1,i_2\}} \ \ \ \ \ (\mathrm{mod} \ 2 \ )$$
and hence $Q(i_1,i_2,i_3)\equiv 0$ (mod $2$) if and only if $iii)$ holds for the triple $(i_1,i_2,i_3)$. Now suppose that at least one of them is not divisible by $4$. WLOG, assume that $k_{\{i_1\}} \not\equiv 0$ (mod $4$). Then $(1+k_{\{p\}})k_{\{i_1,p\}}\equiv 1$ (mod $2$) if and only if $k_{\{p\}}\equiv 0$ (mod $4$) and  $k_{\{i_1\}} \equiv 1$ (mod $4$). Therefore, we have
$$Q(i_1,i_2,i_3)\equiv (1+k_{\{i_2\}})k_{\{i_1,i_3\}}+(1+k_{\{i_3\}})k_{\{i_1,i_2\}}\equiv 0 \ \ \ \ \ (\mathrm{mod} \ 2 \ ).$$ This proves the theorem.
\end{proof}

Since $k_i=\mathrm{deg}^{-}(v_i)+\omega(i)$ and $k_{ij}=M_{ij}+\mathbf{w(v_i,v_j)}\cdot \mathbf{w(v_i,v_j)}+ \mathbf{w(v_j,v_i)}\cdot \mathbf{w(v_j,v_i)}$, we have the following.

\begin{corollary}
	Let $D_M$ be an $\omega$-weighted acyclic digraph associated to a small cover $M$ over $P=\underset{i=1}{\overset{k}{\prod}} \Delta^{n_i}$ with $n_i \geq 3$. Then $w_{3}(M)=0$ if and only if the following conditions hold for vertices of $D_M$:
	\begin{itemize}
		\item[i)] $\mathrm{deg}^{-}(v_i)+\omega(i) \not\equiv 2$ (mod $4$),
		\item[ii)] If $\mathrm{deg}^{-}(v_i)+\omega(i)$ or $\mathrm{deg}^{-}(v_j)+\omega(j)$ is odd then $M_{ij}+\mathbf{w(v_i,v_j)}\cdot \mathbf{w(v_i,v_j)}+ \mathbf{w(v_j,v_i)}\cdot \mathbf{w(v_j,v_i)}\equiv 1$ (mod $2$) if and only if either $\mathrm{deg}^{-}(v_i)+\omega(i)\equiv 0$ (mod $4$) and $\mathrm{deg}^{-}(v_j)+\omega(j) \equiv 1$ (mod $4$) or vice a versa,
		\item[iii)] If $\mathrm{deg}^{-}(v_{i_1})+\omega(i_1)\equiv \mathrm{deg}^{-}(v_{i_2})+\omega(i_2)\equiv \mathrm{deg}^{-}(v_{i_3})+\omega(i_3)\equiv 0$ (mod $4$) then $$\sum_{p\neq q}\Big( M_{i_pi_q}+\mathbf{w(v_{i_p},v_{i_q})}\cdot \mathbf{w(v_{i_p},v_{i_q})}+ \mathbf{w(v_{i_q},v_{i_p})}\cdot \mathbf{w(v_{i_q},v_{i_p})}\Big) \equiv 1 \ (mod \ 2).$$
	\end{itemize}
\end{corollary}

As shown above, when $M$ is a generalized Bott manifold, the Stiefel-Whitney classes of $M$ of dimensions less than equal to $3$ can be written in terms of the dot products of columns of the associated reduced vector matrix $A$. It is natural to ask whether this is true for all dimensions. The following theorem gives an affirmative answer to this question.

\begin{theorem}\label{thm:formula4}
	The fourth Stiefel-Whitney class of a small cover $M$ over $P=\underset{i=1}{\overset{k}{\prod}} \Delta^{n_i}$ modulo $I$ is equal to
	\begin{eqnarray*}w_4(M)&=&\sum \binom{k_{\{i\}}+1}{4}x_{i0}^4 +\sum P_1(i,j) x_{i0}^3x_{j0}+\sum P_2(i,j)x_{i0}^2x_{j0}^2\\&&+ \sum Q(i_1,i_2,i_3))x^2_{i_10}x_{i_20}x_{i_30}+\sum R(i_1,i_2,i_3,i_4)x_{i_10}x_{i_20}x_{i_30}x_{i_40}\end{eqnarray*} where 
	
	\begin{eqnarray*} P_1(i,j)&=&\binom{k_{\{i\}}+1}{3}\cdot \big(k_{\{j\}}+1\big)-\binom{k_{\{i\}}}{2}\cdot k_{\{i,j\}},\\
P_2(i,j)&=&\binom{k_{\{i\}}+1}{2}\cdot \binom{k_{\{j\}}+1}{2}-k_{\{i\}}k_{\{j\}} k_{\{i,j\}}+\binom{k_{\{i,j\}}}{2},\\
Q(i_1,i_2,i_3)&=&\binom{k_{\{i_1\}}+1}{2}\Big(\big(k_{\{i_2\}}+1\big)\big(k_{\{i_3\}}+1\big)-k_{\{i_2,i_3\}}\Big)-k_{\{i_1\}}\Big(\sum_{p\neq 1}k_{\{i_1,i_p\}}\big(k_{\{i_2,i_3\}-\{i_p\}}+1\big)\Big)\\
&&+k_{\{i_1,i_2\}}k_{\{i_1,i_3\}}-k_{\{i_1,i_2,i_3\}},\\
		R(i_1,i_2,i_3,i_4)&=&\Big(\prod_{p=1}^4\big(k_{\{i_p\}}+1\big)\Big)-\sum_{p\neq q}\Big(\big(k_{\{i_p\}}+1\big)\big(k_{\{i_q\}}+1\big)-\frac{k_{\{i_p,i_q\}}}{2}\Big)\cdot k_{\{i_1,i_2,i_3,i_4\} -\{i_p,i_q\}}.
	\end{eqnarray*}
\end{theorem}
\begin{proof} Since the rest can be found similarly, we only provide a proof for the formula for $Q(i_1,i_2,i_3)$. Here $Q(i_1,i_2,i_3)$  is equal to the coefficient of $y_1y_2y_3$ in (\ref{eq:coeffQ}). One can to choose $y_3$ from either of the factors $(1+y_3)^{k_{\{i_3\}}-\underset{p\neq 3}{\sum}k_{\{i_p,i_3\}}+k_{\{i_1,i_2,i_3\}}+1}$, $(1+y_1+y_3)^{k_{\{i_1,i_3\}}-k_{\{i_1,i_2,i_3\}}}$, $(1+y_2+y_3)^{k_{\{i_2,i_3\}}-k_{\{i_1,i_2,i_3\}}}$ or $( 1+y_1+y_2+y_3\Big)^{k_{\{i_1,i_2,i_3\}}}$. Therefore we have
	\begin{eqnarray*}Q(i_1,i_2,i_3)&=&\Big(k_{\{i_3\}}+k_{\{i_1,i_2,i_3\}}+1-\underset{p\neq 3}{\sum}k_{\{i_p,i_3\}} \Big) \cdot \Big[\binom{k_{\{i_1\}}+1}{2}\cdot \big(k_{\{i_2\}}+1\big)-k_{\{i_1\}}k_{\{i_1,i_2\}}\Big]\\&&+\big(k_{\{i_1,i_3\}}-k_{\{i_1,i_2,i_3\}}\big)\cdot\Bigg[\binom{k_{\{i_1\}}}{2}\cdot \big(k_{\{i_2\}}+1\big)-\big(k_{\{i_1\}}-1\big)k_{\{i_1,i_2\}}\Bigg]\\
		&&+\big(k_{\{i_2,i_3\}}-k_{\{i_1,i_2,i_3\}}\big)\cdot\Bigg[\binom{k_{\{i_1\}}+1}{2}\cdot k_{\{i_2\}}-k_{\{i_1\}}k_{\{i_1,i_2\}}\Bigg]\\
		&&+k_{\{i_1,i_2,i_3\}}\cdot \Bigg[\binom{k_{\{i_1\}}}{2}\cdot k_{\{i_2\}}-\big(k_{\{i_1\}}-1\big)\big(k_{\{i_1,i_2\}}-1\big)\Bigg].\end{eqnarray*} Since the sum of the first factors of each term in the  RHS of the equation is $k_{\{i_3\}}+1$, the result easily follows.
\end{proof}

\begin{corollary}Let $M$ be a small cover over $P=\underset{i=1}{\overset{k}{\prod}} \Delta^{n_i}$ with $n_i \geq 4$. Then $w_{4}(M)=0$ if and only if the following conditions hold:
	\begin{itemize}
		\item[i)]$k_{\{i\}}\equiv 0,1,2$ or $7$ (mod $8$),
		\item[ii)] $k_{\{i,j\}}$ must satisfy the following table
		\begin{center}
			\begin{tabular}{c|c|c}
				$k_{\{i\}}$ (mod $8$) & $k_{\{j\}}$ (mod $8$) & $k_{\{i,j\}}$ (mod $4$)\\ \hline 
				$0$&$0$&$1$\\
				$0$&1&$0$ or $1$\\
				$0$&$2$&$1$\\
				$1$&$1$&$2$\\
				$1$&$2$&$2$\\
				$2$&$2$&$3$\\
				-&$7$& $0$				
			\end{tabular}
		\end{center}
		\item[iii)]$k_{\{i,j,l\}}$ must satisfy the following table
		\begin{center}
			\begin{tabular}{c|c|c|c}
				$k_{\{i\}}$ (mod $8$) & $k_{\{j\}}$ (mod $8$) &$k_{\{l\}}$ (mod $8$)& $k_{\{i,j,l\}}$ (mod $2$)\\ \hline 
				$0$&$0$&$0$&$1$\\
				$0$&$0$&1&$k_{\{0,1\}}$\\
				$0$&$0$&$2$&$1$\\
				$0$&$1$&$1$&$k_{\{0,1\}}$\\
				$0$&$1$&$2$&$k_{\{0,1\}}$\\
				$0$&$2$&$2$&$1$\\
				$1$&$1$&$1$&$0$\\
				$1$&$1$&$2$&$0$\\
				$1$&$2$&$2$&$0$\\
				$2$&$2$&$2$&$1$
			\end{tabular}
		\end{center} 
	\end{itemize}
\end{corollary}
\begin{proof} Note that $\displaystyle \binom{k_{\{i\}}+1}{4} \equiv 0$ (mod $2$) if and only if $k_{\{i\}}$ satisfies the condition $i)$. Here $k_{\{i,j\}}$ and $k_{\{i,j,l\}}$ depend on the values of $k_{\{i\}},k_{\{j\}}$ up to modulo $8$, and $k_{\{i\}},k_{\{j\}}$ and $k_{\{l\}}$ up to modulo $8$, respectively. Let $\theta_{i}$ denote the integer between $0$ and $7$ that is congruent to $k_{\{i\}}$ modulo $8$.

Suppose that $w_4(M)=0$. Therefore $P_1(i,j), P_2(i,j), Q(i_1,i_2,i_3)$ and $R(i_1,i_2,i_3,i_4)$ are zero modulo $2$ for all possible combinations. When $\theta_{i}=7$, $P_1(i,j)\equiv k_{\{i,j\}}$ and $P_2(i,j)\equiv k_{\{i,j\}}+\binom{k_{\{i,j\}}}{2}$. This gives that $k_{\{i,j\}} \equiv 0$ (mod $4$) when $\theta_i=7$. When $\theta_{i}=2$, $P_1(i,j)\equiv k_{\{i\}}+1+k_{\{i,j\}}$ (mod $2$) and hence we have $k_{\{i,j\}} \equiv 0$ (mod $2$) when $\theta_{j}= 1,7$ and $k_{\{i,j\}} \equiv 1$ (mod $2$) when $\theta_{j}= 0,2$. Since when $(\theta_{i}, \theta_{j})=(2,2)$, $P_2(i,j) \equiv 1+\binom{k_{\{i,j\}}}{2}$ (mod $2$),  $k_{\{i,j\}}\equiv 3$ (mod $4$). When $\theta_{i}=0$, $P_2(i,j)\equiv 0$ (mod $2$) yields $k_{\{i,j\}}\equiv 0$ or $1$ (mod $4$). In particular we have $k_{\{i,j\}}\equiv 1$ (mod $4$) when $(\theta_{i}, \theta_{j})=(0,2)$. Similarly, when $\theta_{i}=1$, $P_2(i,j)\equiv 0$ (mod $2$) gives $k_{\{i,j\}}\equiv 1$ or $2$ (mod $4$) and hence we have $k_{\{i,j\}}\equiv 2$ (mod $4$) when $(\theta_{i}, \theta_{j})=(1,2)$. 
		
When $\theta_{i}=0$ for all $i\in\{i_1,i_2,i_3,i_4\}$, $R(i_1,i_2,i_3,i_4)\equiv 1+k_{\{i_1,i_2\}}$ (mod $2$) and hence it is zero modulo $2$ if and only if  $k_{\{i_1,i_2\}} \equiv 1$ (mod $2$). Since $k_{\{i_1,i_2\}} \equiv 0$ or $1$ (mod $4$) whenever $\theta_{i}=0$, we have $k_{\{i_1,i_2\}} \equiv 1$ (mod $4$) in this case. Similarly, when $\theta_{i}=1$ for all $i\in\{i_1,i_2,i_3,i_4\}$, $R(i_1,i_2,i_3,i_4)\equiv 0$ (mod $2$) yields $k_{\{i_1,i_2\}} \equiv 2$ (mod $4$) since it is either $1$ or $2$ modulo $4$.

Under these assumptions, when $\theta_i=7$ for one of the $i_1,i_2$ or $i_3$, $Q(i_1,i_2,i_3) \equiv k_{\{i_1,i_2,i_3\}} \equiv 0$ (mod $2$). When $(\theta_{i_1}, \theta_{i_2},\theta_{i_3})=(2,0,0)$, $Q(i_1,i_2,i_3) \equiv 1+ k_{\{i_1,i_2,i_3\}} \equiv 0$ (mod $2$). Similarly, $Q(i_1,i_2,i_3) \equiv 0$ (mod $2$) for $(\theta_{i_1}, \theta_{i_2},\theta_{i_3})=(2,p_1,p_2)$ and $(\theta_{i_1}, \theta_{i_2},\theta_{i_3})=(0,q_1,q_2)$ where $0\leq p_t \leq 2$ and $0 \leq q_t \leq 1$ give the all the remaining restrictions on $k_{\{i_1,i_2,i_3\}}$ and proves the only if part of the theorem. One can easily check that under these restrictions, $w_4(M)=0$. 
\end{proof}

Whenever $m$ is not a power of $2$, the Wu formula can be used to express $w_m$ in terms of lower Stiefel-Whitney classes and their Steenrod squares. Hence one can conclude that whenever the lower dimensional Stiefel-Whitney classes are zero then so is $w_m$ for $m \neq 2^p$ for any $p$. Hence we have the following result.

\begin{corollary} Let $M$ be a small cover over $P=\underset{i=1}{\overset{k}{\prod}} \Delta^{n_i}$ with $n_i \geq 4$ with an associated matrix $A$. Then the first seven Stiefel- Whitney classes of $M$ are zero if and only if $A_i \cdot A_i \equiv 7$ (mod $8$), $A_i \cdot A_j \equiv 0$ (mod $4$) and $k_{\{i,j,l\}}=|\{t|a_{it}=a_{jt}=a_{lt}=1\}| \equiv 0$ (mod $2$) for all $i<j<l$.
\end{corollary}
\begin{proof} By Proposition \ref{prop:firstthree} and the above argument, it suffices to show that if $A_i \cdot A_i \equiv 7$ (mod $8$), $A_i \cdot A_j \equiv 0$ (mod $4$) and $k_{\{i,j,l\}}=|\{t|a_{it}=a_{jt}=a_{lt}=1\}| \equiv 0$ (mod $2$) for all $i<j<l$ then $w_4(M)=0$. This directly follows from Theorem \ref{thm:formula4}. 
\end{proof}

When $m$ is a power of $2$, for all $i+j=m$, $\binom{j-1}{i}$ is always even and hence one can not use the Wu formula to find $w_m$. Considering the results of the paper, we believe that for each $m=2^t$, $k_{\{S\}}$'s where $S$ is a subset of size $t$ of $\{1,2,\cdots,k\}$  will appear as a coefficient of $w_{m}(M)$ and we conjecture the following.

\begin{conjecture}
	Let $M$ be a small cover over $P=\underset{i=1}{\overset{k}{\prod}} \Delta^{n_i}$ with $n_i \geq 2^t$ with an associated matrix $A$. Then the first $2^{t+1}-1$ Stiefel- Whitney classes of $M$ are zero if and only if 	for any $S \subseteq \{1,2,\cdots,k\}$
	of size less than equal to $t+1$, $k_{S}=|\{i| \ a_{si}=1 \ \mathrm{for \ any \ } s\in S \}|$ is congruent to $-1$ modulo $2^{t+1}$ when $|S|=1$ and is congruent to $0$ modulo $2^{t+1-|S|}$, otherwise.
\end{conjecture}

\end{document}